\documentclass[11pt]{amsart}
\usepackage[utf8]{inputenc}
\usepackage{amsmath,amsfonts,amssymb,amscd}
\usepackage{tikz-cd} 
\usepackage{xcolor}
\usepackage{hyperref}
\hypersetup{colorlinks=true, citecolor=blue, linkcolor=violet}
\newcommand{\codim}{\rm codim\,}

\newcommand{\e}{\varepsilon}

\newcommand{\tmop}[1]{\ensuremath{\operatorname{#1}}}
\renewcommand{\Re}{\tmop{Re}}

\newcommand{\C}{\mathbb C}
\newcommand{\CP}{\mathbb{CP}}
\newcommand{\R}{\mathbb R}

\newcommand{\del}{\mathbb \partial}
\newtheorem{theorem}{Theorem}

\newtheorem{prop}[theorem]{Proposition}

\newtheorem{defn}[theorem]{Definition}
\theoremstyle{remark}

\begin{document}
\title{Levi-flats in $\CP^n$: a survey for nonexperts} 
\author{Rasul Shafikov}
\begin{abstract}
This survey paper, aimed at nonexperts in the field, explores various proofs of nonexistence of real analytic Levi-flat hypersurfaces in $\CP^n$, $n>2$. Some generalizations and other related results are also discussed. 
\end{abstract}
\address{Department of Mathematics, the University of Western Ontario, London ON, Canada, N6A 5B7}
\email{shafikov@uwo.ca}
\maketitle

An intriguing open problem in complex geometry is to construct an example or prove the nonexistence of a real analytic (or smooth), 
closed (i.e, compact without boundary), Levi-flat hypersurface in the complex projective plane $\mathbb{CP}^2$. 
This question appeared in the context of a more general problem of the existence of nontrivial minimal sets, see Section~1 for details. Nonexistence of such Levi-flats in 
$\CP^n$ for $n>2$ was proven by several authors, with many generalizations to a more general class of complex manifolds of dimension 
at least 3. The open problem in $\CP^2$ is particularly interesting because it has a natural formulation from many different 
points of view: in the theory of holomorphic foliations as a hypothetical nontrivial minimal set, in several complex variables as a common boundary of two Stein manifolds, and in symplectic geometry, for example, 
 as a limit of Stein fillings of two contact structures corresponding to opposite orientations on the same manifold. 
The purpose of this survey is to collect all relevant information concerning Levi-flat hypersurfaces in projective spaces, some of which exists only as folklore, and outline three (mostly) self-contained proofs of the nonexistence results in $\CP^n$, $n>2$. Two of the three proofs are due to Lins Neto~\cite{LN}; the third one is due to Siu~\cite{Siu}. These proofs are genuinely different, and it is remarkable that all of them fail for $n=2$ for different reasons. 

Although the question makes sense in the smooth category, the emphasis will be on the real analytic case, as it is the most natural formulation 
from the point of view of the holomorphic foliation theory and less technical. The survey is intended for mathematicians from different areas 
of mathematics, so we include some basic foundational material. For a comprehensive overview of the theory of Levi-flat hypersurfaces in complex geometry see a more technical survey by Ohsawa~\cite{Oh4}.

In the first section we discuss generalities of singular holomorphic foliations on projective spaces. In Section~2 
we discuss Levi-flat hypersurfaces, followed by the proofs of nonexistence in Section~3. In the last section we compile some generalizations.

\section{Holomorphic Foliations}
Classical theory of foliations can be found in~\cite{CS}, a comprehensive reference is~\cite{CC, CC2}. For an in-depth discussion of important results and open problems in foliation theory until 1974, see~\cite{La}.
In this section we give general background material on (singular) holomorphic foliations that will be needed for the proofs.

Loosely speaking, a foliation $\mathcal F$ of dimension $k$ (or codimension $n-k$) on an $n$-dimensional manifold $X$ is a decomposition of $X$ into connected manifolds $L_\nu$, $\nu\in \Lambda$, of 
dimension $k$, called the leaves of $\mathcal F$, such that in a coordinate chart 
$(U, \phi)$ of $X$,  $\phi (U) =D^k \times D^{n-k} \subset \R^k \times \R^{n-k}$, each $L_\nu \cap U$ is an at most countable 
union of the sets of the form $D^k \times \{x\}$, $x\in D_{n-k}$. The sets $\phi^{-1}(D^k \times \{x\})$ are called the plaques 
of $U$. On the overlaps the requirement is that 
the transition map has the form
\begin{equation}\label{e.foli}
h(x,y) = (h_1(x,y), h_2(y)), \ (x,y) \in \R^k \times \R^{n-k} .
\end{equation}
The smoothness of $h$ determines the smoothness of the foliation: $C^k$-, $C^\infty$-smooth, or real analytic.
Holomorphic foliations are defined in a similar way. 

For codimension 1 foliations, the local representation in a coordinate chart gives rise to a local first integral: 
a nonconstant smooth function that is constant on the leaves (plaques) of the foliation. If the foliation is holomorphic (resp. real analytic), then the first integral can be chosen to be holomorphic (resp. real analytic).

For an open set  $U\subset \C^n$, $n>1$, denote by $\Omega^1(U)$ the space holomorphic 1-forms, i.e,  $\omega \in \Omega^1(U) $ if 
$\omega=\sum_{k=1}^n a_k(z)dz_k$ with  $a_k(z) \in \mathcal O(U)$. Let $\omega \in \Omega^1(U) $ and 
$$
Y = \{\omega=0\}=\{ z \in U : a_1(z)=\dots=a_n(z) = 0\}. 
$$
Then $\text{ker}\,\omega(z) \subset T_z U$ defines a distribution of complex hyperplanes on $U\setminus Y$. If $\omega$ is integrable, i.e., 
if $\omega \wedge d\omega = 0$, then the distribution $\text{ker}\,\omega(z)$ is involutive, and so by the Frobenius theorem 
$U\setminus Y$ is foliated by submanifolds that are tangent to  $\text{ker}\,\omega(z)$ at every point. By the Levi-Civita theorem, the leaves of the foliation are complex manifolds of dimension $n-1$. If $\dim Y = n-1$, we may divide all $a_k(z)$ by an appropriate function vanishing  on $Y$, and so 
we may assume that $\dim Y < n-1$. Thus, we obtain on $U$ a singular holomorphic foliation $\mathcal F$ whose singular locus
$\text{Sing}\,\mathcal F :=Y$ has (complex) codimension at least 2. 

Let now $X$ be a complex manifold, $\dim X >1$, with an open cover $\cup_{j\in J} U_j=X$ by open sets. Suppose that for every $U_j$ there is 
an integrable 1-form $\omega_j\in \Omega^1(U_j)$, $\codim\{\omega_j =0\}\ge 2$, such that if $U_j \cap U_k \ne \varnothing$, then $\omega_j = g_{jk}\omega_k$ in
$U_j \cap U_k$ and $g_{jk} \in \mathcal O^*(U_j \cap U_l)$ (nonvanishing holomorphic function). The local foliations in $U_j$ and their
singularities glue together and we obtain a singular holomorphic foliation $\mathcal F$ on $X$ of codimension $1$ with a singular locus of codimension at least~2. Note that closed leaves of $\mathcal F$ and $\text{Sing}\,\mathcal F$ are (closed) complex subvarieties of $X$. 
The cocycle $\{g_{jk}\} \subset \mathcal O^*(U_j \cap U_k)$ defines a line bundle on $X$, called the normal bundle of $\mathcal F$ and denoted by~$N_{\mathcal F}$.  

A holomorphic foliation of codimension 1 can be defined via a holomorphic submersion into a Riemann surface; the leaves of the foliation are then the level sets of the submersion. 
In dimension 2 a foliation can also be defined by integral curves of (locally defined, singular) holomorphic vector fields. 
This is equivalent to the description by 1-forms, indeed, a foliation defined in local coordinates by a holomorphic form $\omega = P(z) dz_2 - Q(z) dz_1$ can also be defined by a vector field $V = P(z)\frac{\partial}{\partial z_1} + Q(z)\frac{\partial}{\partial z_2}$.

Consider now a special case of a singular holomorphic foliation $\mathcal F$ on the complex projective space $\CP^n$. If $\pi: \C^{n+1} \to \CP^n$ is the natural projection, then  $\pi^* ({\mathcal F})=\tilde{\mathcal F}$ is a foliation on $\C^{n+1}$. It is known 
(e.g.,~\cite[p.~577]{CLN6}) that $\tilde{\mathcal F}$
is defined by just one integrable holomorphic 1-form on $\C^{n+1}$
$$
\tilde\omega = \sum_{j=0}^n a_j(z) dz_j, 
$$
where $a_j(z)$ are homogeneous polynomials of the same degree $d+1$ that in addition satisfy the so-called Euler condition
$$
\sum_{j=0}^n z_j a_j(z) \equiv 0.
$$
The integer $d\ge 0$ is called the degree of $\mathcal F$. One can show that $d$ is equal to the number of points (counting with multiplicity) where a generic $\CP^1 \subset \mathbb \CP^n$ is tangent to the leaves of $\mathcal F$. For a specific example, consider the union of all complex lines in $\C^2$ passing through the origin. This gives a foliation on $\CP^2$ singular at the origin in $\C^2$. The corresponding 1-form $\tilde\omega$ in $\C^3$
is $z_1 dz_2 - z_2 dz_1$ in coordinates $(z_0, z_1, z_2)$. Thus, the degree of this foliation is  0. In general, foliations of degree 0 and 1 always contain
algebraic leaves, but for any $d>2$ there exist foliations that contain no algebraic leaves and every leaf is dense in $\CP^n$, 
see~\cite{Jo}.

\bigskip

A holomorphic foliation $\mathcal F$ on $\CP^2$ can be given both in terms of holomorphic 1-forms and holomorphic vector fields 
and has at most a finite set of singular points. Let $p\in \text{Sing}\,\mathcal F$ and let $V$ be a holomorphic vector field in a neighbourhood of $p$ that 
defines $\mathcal F$ there. Let $A = DV(p)$ be the linear part of $V$ at $p$, and assume that $DV(p)$ is nondegenerate. Then the Baum-Bott index of $\mathcal F$ at $p$ is  a complex number defined as 
$$
BB(\mathcal F, p) = \frac{(\text{trace A})^2}{\text{det}A}.
$$
As a basic example consider a vector field V in $\mathbb C^2$ such that $V(0)=0$ and $DV(0)$ is nondegenerate. Assume in addition that the foliation $\mathcal F$ generated by $V$ admits a holomorphic first integral $f$, which means that $V(f) =0$, and also assume that $df(0)=0$, and $(D^2f)(0)$, the matrix of second order derivatives, is nondegenerate. Such singularities of $\mathcal F$ are called Morse-type. Then, by the holomorphic Morse lemma (see, e.g.,~\cite{Zo}), there exists a holomorphic coordinate system
$z=(z_1,z_2)$ such that $f(z)=z_1 z_2$. If $V(z) = P(z) \frac{\partial}{\del z_1}+Q(z)\frac{\del}{\del z_2}$, then
$V(f) = z_2 P(z) + z_1 Q(z) \equiv 0$, which implies   
$$
V (z) = \tilde f(z) \left(z_1\frac{\partial}{\del z_1}-z_2\frac{\del}{\del z_2} \right), \ \tilde f(0)\ne 0.
$$
In particular, this means that $BB(\mathcal F, 0) = 0$. 

The following result for  holomorphic foliations in projective spaces is a consequence of a general result of
Baum-Bott~\cite{BB}.

\begin{prop}[Baum-Bott Index formula]\label{p.bb}
Let $\mathcal F$ be a holomorphic foliation of degree $d$ on $\CP^2$. Then
\begin{equation}\label{e.bb}
\sum_{p \in \text{Sing}\,\mathcal F} BB(\mathcal F, p) = (d+2)^2.
\end{equation}
\end{prop}

The proof can be found in Lins Neto~\cite{LN1}. For further connections of the Baum-Bott index and other indices see Brunella~\cite{Br}.

In particular, since the right-hand side in~\eqref{e.bb} is always positive, it immediately follows that there do not exist nonsingular holomorphic foliations on $\CP^2$. Furthermore, the following holds.

\begin{prop}[Lins Neto~\cite{LN}]\label{p.cod2}
For any holomorphic foliation on $\CP^n$, $n>1$, the singular locus $\text{Sing}\,\mathcal F$ contains an irreducible component 
of codimension~2.
\end{prop}

\begin{proof}
The result holds for $n=2$ as observed above. Suppose that $n>2$ and $\mathcal F$ is a holomorphic foliation on $\CP^n$ with 
$\codim \text{Sing}\,\mathcal F >2$. Then there exists a generic 2-plane $E = \CP^2 \subset \CP^n$ such that 
$E \cap \text{Sing}\,\mathcal F =\varnothing$, $E$ is not contained in any leaf of the foliation $\mathcal F$, 
and the set of tangencies of 
$E$ with the leaves of $\mathcal F$ has codimension 2 in $E$. In fact, one can show (see~\cite{CLNS} or~\cite{Mi2}) 
that the set of $E$ with these properties 
is dense in the Grassmannian $\text{Gr}_{\mathbb C}(3,n+1)$. The inclusion map $\iota: E \subset \CP^n$ induces a holomorphic foliation 
$\mathcal G = \iota^*(\mathcal F)$ on $E$.

The foliation $\mathcal G$ is singular at points where $E$ is tangent to $\mathcal F$, and so by the assumption on $E$,  
$\text{Sing}\,\mathcal G$ is a finite set. Let $p\in \text{Sing}\,\mathcal G$. Since $p$ is a regular point of $\mathcal F$, 
the foliation $\mathcal F$ admits a local first integral, say $f$, and therefore, $g = f \circ \iota$ is a first integral for $\mathcal G$.
But since $p$ is a singular point of $\mathcal G$, $dg(p) =0$. After an additional small perturbation of $E$ we may assume that $D^2g(p)$ is nondegenerate, and therefore, $p$ is a Morse-type singularity of $\mathcal G$. This implies that $BB(\mathcal G, p)=0$,
and we obtain a contradiction with the index formula~\eqref{e.bb}.
\end{proof}

Recall that a smooth real-valued function $\rho$ on a complex manifold $X$ is called strictly plurisubharmonic if 
$i\partial\overline\partial\,\rho$ is a positive-definite $(1,1)$-form. This is equivalent to the condition that  
the complex Hessian defined by~\eqref{e.levi} is positive-definite for all $p\in X$ and $w\in T_p X \setminus \{0\}$.
A complex manifold $X$ is called Stein if one of the following equivalent conditions hold:

(i) $X$ is holomorphically convex and holomorphically separable (see, e.g.,~\cite{For} for relevant definitions);

(ii) $X$ admits a strictly plurisubharmonic exhaustion function $\rho$, i.e., $\rho (-\infty, c]$ is compact for any $c\in \mathbb R$; 

(iii) $X$ admits a proper holomorphic embedding into $\C^N$ for some $N>0$.

\noindent For our purposes only properties (ii) and (iii) will be relevant. A domain 
$D\subset \C^n$ is called locally Stein if for any point $p\in b D$ there is a neighbourhood $U_p$ such that 
$U_p \cap D$ is Stein. The solution of the Levi problem for domains in $\mathbb C^n$ implies that if an (open) 
domain $D \subset \mathbb C^n$ is locally Stein, then it is Stein. 

We conclude this section with a brief discussion of minimal sets. Given a (smooth) foliation $\mathcal F$ on a (smooth) 
manifold~$X$, a subset $Y$ of $X$ is called invariant or saturated with respect to $\mathcal F$ if for every point $p\in Y$ 
we have $L_p \subset Y$, where $L_p$ is the leaf of $\mathcal F$ passing through $p$. A set $Y\subset X$ is called minimal 
if $Y$ is nonempty, closed, invariant with respect to $\mathcal F$, and satisfies the following property: if $Y' \subset Y$ is 
closed and invariant, then either $Y' = \varnothing$ or $Y' = Y$. 
For example, a closed leaf is a minimal set. The closure of a (nonclosed) leaf is an invariant set but may not be minimal. 
It is not hard to see that if a minimal set $Y$ contains an open subset of $X$, then $Y=X$, 
and if  $X$ is compact, then any foliation $\mathcal F$ on $X$ has a nontrivial minimal set. In general, the geometry of 
a minimal set can be complicated, for example, one may construct an example of a minimal set $Y$ 
of a codimension 1  foliation $\mathcal F$ with the property that if  $\gamma$ is a real arc transverse to $\mathcal F$ such that 
$\gamma$ has a nonempty intersection with $Y$, then $Y\cap \gamma$ is homeomorphic to the Cantor set.

In the context of singular holomorphic foliations one can impose an additional condition: a minimal set $Y$ of a singular holomorphic foliations $\mathcal F$ is called nontrivial if $Y \cap \,\text{Sing}\,\mathcal F = \varnothing$. 
Clearly, a nontrivial minimal set exists iff there exists a leaf of the foliation whose closure does not contain any singular point of 
$\mathcal F$. Camacho, Lins Neto, and Sad~\cite{CLNS1} studied nontrivial minimal sets for singular holomorphic foliations 
on $\CP^2$ and, in particular, proved the following.

\begin{theorem}\label{t.min}
Let $\mathcal F$ be a singular holomorphic foliation on $\CP^2$. Suppose that a nontrivial minimal set $Y$ exists. Then

(i) $Y$ is unique;

(ii) every one dimensional algebraic variety in $\CP^2$ intersects $Y$.

(iii) if $L$ is a leaf of $\mathcal F$ such that $\overline L \cap \text{Sing}\,\mathcal F = \varnothing$, then $L$ accumulates at $Y$, i.e., $\overline L \setminus L = Y$.

\end{theorem}

\begin{proof}
Since $Y$ is an invariant set of $\mathcal F$, every point $p \in Y$ belongs to $Y$ together with the leaf of $\mathcal F$ passing through $p$. This means that $\CP^2 \setminus Y$ is locally Stein at every boundary point, and so by the result of  Fujita~\cite{Fu} and
Takeuchi~\cite{Ta},  $\CP^2 \setminus Y$ is a Stein manifold.
From this, part (ii) immediately follows because Stein manifolds do not admit compact subvarieties of positive dimension. 
Suppose now that $Y'$ is another nontrivial minimal set, $Y \ne Y'$. Then $Y \cap Y' = \varnothing$, as otherwise the intersection would be an invariant set, contradicting minimality of $Y$ and $Y'$. Therefore, $Y'$ is a compact subset of $\CP^2 \setminus Y$. Consider a proper holomorphic embedding $\phi: \CP^2 \setminus Y \hookrightarrow \C^N$, 
which exists since $\CP^2 \setminus Y$ is Stein.
There exists a large ball $B(0,R)$ in $\C^N$ that contains $\phi(Y')$. By shrinking the radius of the ball we obtain $R_0\in\R$ such that the sphere
$bB(0,R_0)$ touches $\phi(Y')$ at some point $q$. Then the function $|z|^2$ on $\C^N$, when restricted to the leaf of the foliation passing through $q$, attains a maximum, but this contradicts the Maximum principle for plurisubharmonic functions. This contradiction shows that $Y$ is unique. A similar argument can be used to prove part~(iii).
\end{proof}

Camacho, Lins Neto, and Sad~\cite{CLNS1} also proved that in the space of singular holomorphic foliations on $\CP^2$ 
the subset of foliations without a nontrivial minimal set is open and nonempty. In fact, there are no known examples of nontrivial 
minimal sets. Theorem~\ref{t.min} holds for all $n>1$, however, Lins Neto~\cite{LN} 
proved that there are no nontrivial minimal sets for any holomorphic foliation in $\CP^n$ for $n>2$, see Section~\ref{s.LN1}.
We will also see in that section that the existence of a real analytic Levi-flat hypersurface 
$M \subset \CP^n$ implies the existence of a singular holomorphic foliation on $\CP^n$ for which $M$ is an invariant 
subset, and therefore $M$ contains a nontrivial minimal set. We discuss Levi-flats next.

\section{Levi-flat hypersurfaces}

A smooth (resp. real analytic) hypersurface $M$ in a complex manifold $X$, $\dim X = n >1$, is locally defined as the zero locus of 
a real-valued smooth (resp. real analytic) function $\rho$ with the nonvanishing gradient. An important biholomorphic 
invariant of $M$ is its Levi-form which locally can be defined as follows. Suppose $M \cap U = \{\rho = 0\}$, where $\rho$ is a 
smooth or real analytic function in an open neighbourhood $U$  of a point $p\in M$, and $\nabla \rho \ne 0$. 
Consider the Hermitian quadratic form defined by 
\begin{equation}\label{e.levi}
\mathcal L_\rho (p, w) = \sum_{j,k = 1}^n 
\frac{\del^2 \rho}{\del z_j \overline\del z_k}(p) \, w_j \overline w_k  , \ \ p \in M, \ w \in \mathbb C^n \cong T_p X.
\end{equation}
Recall that the complex tangent space $H_pM$ at a point $p\in M$ is defined as $H_p M = T_p M \cap J T_p M$, 
where $J$ is the standard almost complex structure on $X$; it is the $(n-1)$-dimensional complex linear subspace of $T_p M$. 
The restriction of $\mathcal L_\rho(p)$ to $H_p(M)$ (i.e., by restricting $w\in H_pM$) is called the {\it Levi-form} of $M$ at $p$. One can show that the signature of the Levi form is independent of the choice of the defining function $\rho$ and is invariant under biholomorphic mappings. 

A real hypersurface $M$ locally divides the ambient manifold into 2 connected open components: 
$U^+=\{\rho>0\}$ and $U^-=\{\rho < 0\}$. We say that $U^-$ is Levi-pseudoconvex at a point 
$p\in M$ if the Levi form of $M$  at $p$ is positive-semidefinite, i.e., $\mathcal L_\rho(p,w) \ge 0$ for all $w\in H_p M$. This is equivalent to saying that all the eigenvalues of the Levi form are nonnegative. 
For domains with smooth boundary Levi-pseudoconvexity of the boundary is equivalent to the domain 
being locally Stein.

\begin{defn}
A $C^2$-smooth real hypersurface $M\subset X$ is called Levi-flat if one of the following equivalent conditions hold:

(i) $M$ is foliated by complex hypersurfaces;

(ii) The Levi-form of $M$ vanishes identically;

(iii) The distribution of complex tangents $\{H_p M,\,p \in M\}$ is involutive (in the sense of Frobenius);

(iv) $M$ locally divides $X$ into two locally Stein domains. 
\end{defn}

The implication (ii) $\Leftrightarrow $ (iv) follows from the general theory of pseudoconvex domains; 
(i) $\Leftrightarrow$ (iii) is the Frobenius theorem; and 
(ii) $\Leftrightarrow$ (iii) can be proved by a direct computation. 
Using conditions (i) and (iv) it is possible to define Levi-flatness for hypersurfaces 
of class $C^1$ or even locally Lipschitz graphs. 

The foliation of $M$ by complex hypersurfaces is called the Levi foliation of $M$ and will be denoted by $\mathcal L$. Note that if the hypersurface $M$ is real analytic, then $\mathcal L$ is a real analytic foliation on $M$ of real codimension~1. Simple examples of Levi-flat hypersurfaces are $\{x_n=0\} \subset \C^n$ with coordinates $z=(z_1, \dots, z_n = x_n + i y_n)$; if $f \in \mathcal O(X)$
and $df\ne 0$, then $\{\Re f = \text{const}\}$ is a Levi-flat hypersurface. For compact examples, consider $M = \CP^1 \times \gamma$, where $\gamma$ is a simple smooth closed curve in $\CP^1$. Then $M$ is foliated by $\CP^1 \times \{p\}$, $p\in \gamma$, and therefore it is a closed Levi-flat hypersurface in $\CP^1 \times \CP^1$. 

The following proposition is due to H.~Cartan~\cite{Ca}.

\begin{prop}\label{p.1}
If $p \in M\subset \mathbb C^n$ is a real analytic Levi flat hypersurface, then there exists a neighbourhood 
$U$ of $p$ in $\mathbb C^n$ and a biholomorphic map $f: U \to U'$ such that 
$$
f(M\cap U) = \{x_n = 0\} \cap U' ,
$$
where $z=(z_1, \dots, z_n = x_n + i y_n)$ is a holomorphic coordinate system on $U'$.
\end{prop}

\begin{proof}
We write $z=(z', z_n)$. Then, after a translation and rotation, and using the Implicit Function theorem, we can arrange $p=0$, and
$$
M = \left\{ y_n = r(z', \overline z', x_n), \ r(0) = 0, dr(0)= 0 \right\},
$$
where $r(z', \overline z', \Re z_n)$ is a real-valued real analytic function near the origin.
Let $f$ be a purely imaginary-valued real analytic first integral of the Levi foliation $\mathcal L$ on $M$, defined in a neighbourhood of $0$. 
We may further assume that $\frac{\partial f}{\partial x_n}(0)\ne 0$. 
Then $f$ is constant on the leaves of 
$\mathcal L$, and therefore it is a CR-function on $M$. It is well-known that any real analytic CR function $f$ extends as a holomorphic 
function $F$ to a neighbourhood of $M$, see~\cite{To}. Then $(z', z) \mapsto (z', F(z))$ has nondegenerate differential at $0$ and so 
it is the required change of coordinates.
\end{proof}

The following proposition is essentially a consequence of the previous result. It was first proved by Rea~\cite{Re} in a 
more general setting of Levi-flats of general codimension.

\begin{prop}\label{p.folext}
If $M$ is a nonsingular real analytic Levi-flat hypersurface in a complex manifold~$X$, then there exists a neighbourhood 
$V$ of $M$ in $X$ such that the Levi foliation $\mathcal L$ extends to a holomorphic foliation $\mathcal F$ on $V$.
\end{prop}

\begin{proof}
By Prop.~\ref{p.1}, for every point $p\in M$ there exists a neighbourhood $U$ and a biholomorphic map $f: U \to U'$ that sends 
$M \cap U$ onto $M'=\{x_n=0\}\cap U'$. The holomorphic foliation on $U'$ given by complex hyperplanes 
$z_n = c$, $c \in \C$, extends the Levi foliation on $M'$, and therefore, the pullback of this foliation under $f$ gives an extension 
of the Levi foliation on $M \cap U$ to a holomorphic foliation on $U$. 

Now choose an open cover $\{U_j\}$ of $M$ by such open sets, so that on each $M \cap U_j$ the Levi foliation extends 
to a holomorphic foliation via  a local biholomorphism $f_j : U_j\cap M \to U'_j \cap \{x_n =0\}$. It remains to show that these
extensions agree on the overlaps. This can be seen as follows: if $U_j \cap U_k \ne \varnothing$, the map 
$$
h=f_k \circ f_j^{-1}: f_j (U_j \cap U_k) \to f_k (U_j \cap U_k)
$$
is a biholomorphism that sends complex hyperplanes in $U_j \cap \{x_n =0 \}$ to complex hyperplanes in 
$U_k \cap \{x_n=0\}$. Writing  $h = (h_1, \dots, h_n)$, the above conditions means that 
$\Re h_n(z',c)\equiv \text{const}$ for all fixed $c \in \R$, which implies that $h_n(z',c)\equiv \text{const}$. This implies that
$h_n$ is independent of $z'$, and therefore, 
$$
h(z', z) = (h'(z), h_n (z_n)).
$$
This is a holomorphic version of~\eqref{e.foli}, which proves the required statement.
\end{proof}

For some $C^\infty$-smooth but not real analytic Levi-flats, the Levi foliation may still extend holomorphically to a neighbourhood 
of the Levi-flat, for example, consider a hypersurface $y_n = f(x_n)$, for a smooth but not analytic $f$. 
But in general, the Levi foliation on a smooth but not analytic Levi-flat $M$ does not extend as a holomorphic foliation to any neighbourhood of $M$, for a specific example see~\cite{Re}.
More generally, one may consider singular real analytic Levi-flat hypersurfaces, these are real analytic sets of dimension $2n-1$ whose regular locus of dimension $2n-1$ is foliated by complex hypersurfaces. It turns out that near singular points, the Levi foliation may not extend to the ambient neighbourhood even as a singular holomorphic foliation, but under some general conditions there is an extension 
as a singular holomorphic web, see~\cite{Bru1, SS}.

Another immediate consequence of Prop.~\ref{p.1} is that a real analytic Levi-flat hypersurface $M$ is locally defined as a zero locus 
of a pluriharmonic function. Further, Barrett~\cite{Ba} showed that a global pluriharmonic defining function exists iff the 
holomorphic foliation in a neighbourhood of $M$ that extends the Levi foliation on $M$ is defined by a nonvanishing real 
analytic closed 1-form.

\bigskip

Next we discuss some topological properties of Levi-flats. Any real hypersurface locally divides the ambient manifold into two connected components. The global result is the following.

\begin{prop}\label{p.or}
Let $M$ be a smooth connected closed real hypersurface in $\CP^n$, $n \ge 1$. Then $M$ is orientable 
and divides $\CP^n$ into two connected components. 
\end{prop}

The topological version of this on $\R^2$ is known as the Jordan curve theorem, and for hypersurfaces in $\R^n$ as the Jordan-Brouwer Separation theorem. 

\begin{proof}
The proof below is based on Samelson~\cite{Sa}, see also Mishchenko~\cite{Mi1}.
We first show orientability. Arguing by contradiction, suppose that $M$ is not orientable. Then there exists a closed path $\tau$ 
on $M$ starting and terminating at a point $p \in M$ such that a normal unit vector to $M$ at $p$ when transported along the curve $\tau$ returns to $p$ pointing in the opposite direction. By taking a point $p'$ on the initial normal to $M$ at $p$, sufficiently close to $p$, 
and moving $p'$ on the normal along the curve $\tau$ upon the return to $p$ we obtain a point $p''$ on the normal to $M$ but on the other side of $M$. This gives us a curve starting at $p'$ and terminating at $p''$ that does not intersect $M$. We now connect $p'$ and $p''$ with a segment of a straight line in the normal to $M$ at $p$. After smoothing near $p'$ and $p''$, we obtain a smooth 
closed curve  $\gamma$ in $\CP^n$ which transversely intersects $M$ at exactly one point $p$.  Since $\CP^n$ is simply-connected, there exists a smooth map $h : D^2 \to \CP^n$ such that $h(bD^2) = \gamma$. Note that in a neighbourhood of $\gamma$ the map $h$ is transverse
to $M$, and therefore, by the Transversality theorem (see, e.g.,~\cite[Ch.3]{Hi}) there exists a small $C^2$-perturbation of $h$, preserving $h$ near the boundary of $D^2$, such that $h$ is transverse to $M$. Then by transversality, $h^{-1}(M)$ is a union of closed 
smooth curves in $D^2$; these are either compact curves in $D^2$ or have 2 terminal points on $bD^2$. But note that the curve in 
$h^{-1}(M)$ with the end point $h^{-1}(p)$ does not have the other end on $bD^2$ because by construction $\gamma$ intersects $M$ only at $p$. This contradiction proves that $M$ is orientable.

Next we show that $M$ divides $\CP^n$ into 2 connected components. Since $M$ is orientable and connected, there is a neighbourhood $U$ of $M$ such that $M$ divides $U$ into exactly 2 connected components, say, $U'$ and $U''$.
 Suppose that $p'\in U'$ and $p''\in U''$ are two points on the opposite side of the normal line to $M$ at some point $p$. 
 The same argument as above shows that $p'$ and $p''$ cannot be connected by a smooth curve in $\CP^n \setminus M$. 
 Thus, $\CP^n\setminus M$ consists of at least two path connected components: $X'$ containing $U'$ and $X''$ containing 
 $U''$. Since $\CP^n$ is connected, it is easy to see that $\CP^n\setminus M = X' \cup X''$.
\end{proof}

Let $M$ be a closed smooth Levi-flat hypersurface in $\CP^n$, $n\ge 2$. Then by Prop.~\ref{p.or}, $M$ divides $\CP^n$ into two connected complex manifolds $X_1$ and $X_2$, both are locally Stein near every boundary point.   In general, it is not true that a relatively compact domain of a complex manifold that is locally Stein near every boundary point is a Stein manifold. But for domains in $\CP^n$, 
Fujita~\cite{Fu} and Takeuchi~\cite{Ta} 
 proved that a domain with nonempty locally Stein boundary is Stein. To summarize, we obtained the following.

\begin{prop}\label{p.st}
A smooth closed Levi-flat manifold divides the ambient $\CP^n$ into two Stein manifolds.
\end{prop}

The proof of nonexistence of Levi-flats in $\CP^n$, $n>2$, essentially boils down to showing that this is impossible. 
Without the assumption of smoothness, it is easy to divide the projective space into two Stein manifolds, for example, 
one can take the closure in $\CP^2$ of $\{x_2=0\}\subset \C^2$. Further, there are examples of compact complex surfaces 
which can be divided into Stein parts by smooth Levi-flats. Indeed, Ohsawa~\cite{Oh82} constructed an example of a Levi-flat hypersurface in the product of 
$\CP^1$ and an elliptic curve such that the complement of the Levi-flat is the union of two Stein manifolds. Further, Nemirovski~\cite{Ne} constructed a family of examples of compact algebraic and nonalgebraic (Hopf surfaces) complex surfaces 
that admit closed smooth Levi-flat hypersurfaces. In these examples a Levi-flat also divides the complex surface into two Stein manifolds.

On the other hand, nonexistence of smooth Levi-flats in Stein manifolds is elementary.

\begin{prop}\label{p.st2}
There does not exist a closed smooth Levi-flat hypersurface in a Stein manifold of dimension $\ge 2$.
\end{prop}

\begin{proof}
Any Stein manifold $X$ can be properly embedded into $\C^N$ for some $N>0$. Under such an embedding, the image of a closed Levi-flat $M$ is a compact subset of $\C^N$, in particular, is contained in some large ball $B(0,R)$, $R>0$. We may decrease $R$ until for some 
$R_0>0$ the sphere $bB(0,R_0)$ touches $M$ at some point $p$. At this point the restriction of the function $|z|^2$ 
to the leaf through $p$ of the Levi foliation on $M$  attains a maximum, which contradicts the Maximum principle for plurisubharmonic functions.
\end{proof}

As a consequence we also obtain that a Levi-flat hypersurface in $\CP^n$ cannot be globally defined by a pluriharmonic function. Indeed, if $M = \rho^{-1}(0)$ for some pluriharmonic function $\rho$ in a neighbourhood of $M$, then for small $c \ne 0$, the hypersurface 
$M_c = \{\rho^{-1}(c)\}$ is also Levi-flat. But then, by Prop.~\ref{p.st}, $M_c$ is contained in a Stein manifold, which is impossible by Prop.~\ref{p.st2}.

\bigskip

In the remaining part of this section we collect some general facts about Levi-flat hypersurfaces in $\CP^n$. By nonexistence, these results are (potentially) nontrivial only for $n=2$, and describe the empty set for $n>2$. Nevertheless, we formulate the results for a general $n$.

Some topological properties of Levi-flat hypersurfaces can be deduced from the result of Hill-Nacinovich~\cite{HN1}. Their result holds for a general class of weakly $q$-concave projective CR manifolds of general CR dimension and codimension, which, in particular, includes Levi-flat hypersurfaces, but for simplicity we formulate it for Levi-flats only. This result can be considered as an analogue of the classical Lefschetz Hyperplane theorem.

\begin{theorem}[\cite{HN1}]
Let $M \subset \CP^n$ be a closed smooth Levi-flat hypersurface, $n>1$. Let $P_0 \cong \CP^{n-1}$ be a complex hyperplane,
and $M_0 = M \cap P_0$ be the hyperplane section. Then the natural homomorphism
$$
H^j (M, \mathbb Z) \mapsto H^j (M_0, \mathbb Z)
$$
is an isomorphism for $j < n-2$ and injective for $j=n-2$. Further, the natural homomorphism
$$
H_j (M_0, \mathbb Z) \mapsto H_j (M, \mathbb Z)
$$
is an isomorphism for $j < n-2$ and is surjective for $j=n-2$. Finally,
$$
\pi_j (M, M_0) = 0, \ \ j < n-1,
$$
and
$$
\pi_j (M_0) \mapsto \pi_j (M)
$$
is an isomorphism for $j<n-2$ and injective for $j=n-2$.
\end{theorem}

\bigskip

Similar results were obtained in \cite{LN}, and~\cite{NW}. Another application of Prop.~\ref{p.or} is the following result.

\begin{prop}\label{p.compactleaf}
A smooth Levi-flat hypersurface $M \subset \CP^n$, $n \ge 2$, cannot contain a compact leaf.
\end{prop}

\begin{proof}
We follow Mishchenko~\cite{Mi1}. Without loss of generality assume that $n=2$. Suppose that 
$S\subset M$ is a compact leaf, hence a smooth projective curve. Since $M$ is orientable, we may consider a small 
perturbation of $S$ by translating it along the positive normal direction to $M$. The perturbed $\tilde S \subset \CP^2$ 
is a smooth real surface with $S \cap \tilde S = \varnothing$. But this contradicts the fact 
that any complex projective curve in $\CP^2$ has positive self-intersection index.
\end{proof}

A hypersurface in $\C^n$ is called real algebraic if it can be defined as a zero locus of a real polynomial. The closure of such 
a hypersurface is a real algebraic hypersurface in $\CP^n$, possibly singular at infinity.

\begin{prop}
There does not exist smooth real algebraic Levi-flat hypersurface in $\CP^n$.
\end{prop}

\begin{proof}
Arguing by contradiction, suppose such $M$ exists and is given by a real polynomial $P(z,\overline z)$. We complexify $P$ by replacing $\overline z$ with an independent variable $\overline w$. For every fixed $w \in \C^n$ the algebraic hypersurface 
$Q_w = \{ z\in \C^n : P(z,\overline w)=0\}$ is called the Segre variety of $w$ associated with~$M$. The family of Segre varieties is, in fact, a biholomorphic invariant of $M$, see, e.g.,~\cite{PSS}. Further, if $w\in M$, then $w\in Q_w \subset M$, and $Q_w$ agrees with the leaf of the Levi foliation passing through $p$, see~\cite{SS}. This means that $M$ contains a closed leaf, which contradicts Prop.~\ref{p.compactleaf}. Another way to reach a contradiction is simply to observe that the leaves of $M$ are projective varieties that must have a nonempty intersection, which implies that $M$ must have singularities.
\end{proof}

Segre varieties can be defined locally near any point on any real analytic hypersurface, including singular points. If $M$ is Levi-flat and is tangent to a holomorphic foliation, then near the singularities the family of Segre varieties can be used to study the singularities of foliations, for example existence of a germ of a Levi-flat tangent to a holomorphic foliation is equivalent to the existence of a local meromorphic first integral, see~\cite{BuGo, CerNeto, Bru2}.

\bigskip

Let us now draw the connection between Levi-flat hypersurfaces and nontrivial minimal sets. Roughly speaking, the holonomy of a 
leaf $L$ of a foliation is a map $\pi_1(L,p) \to \text{Diff} (\Sigma,p)$, $p\in L$, from the fundamental group of $L$ 
into the group of germs of diffeomorphisms of the transverse section $\Sigma$ at $p$ to the foliation, see, e.g.,~\cite{CS} or~\cite{CC} for a precise definition.
Cerveau~\cite{Cer} proved the following.
\begin{theorem}\label{t.S}
Suppose that $\mathcal F$ is a holomorphic foliation on $\CP^n$, $n\ge 2$, and $\mathcal M$ is a nonempty nontrivial 
minimal set. Then there exists a point $p\in \mathcal M$ such that the leaf $L_p$ of $\mathcal F$ has hyperbolic holonomy,
i.e., its holonomy group contains a map $f$ with $|f'(0)|<1$. Further, the following dichotomy holds: either 

(1) $\mathcal M$ is a real-analytic Levi-flat hypersurface, or 

(2) $\text{Hol}: \pi_1 (L_p, p) \mapsto \text{ Diff} (\C,0) $ is a linearizable abelian group. 
\end{theorem}

It is this result that prompted the interest in Levi-flat hypersurfaces on projective spaces. By Lins Neto~\cite{LN} (see the end of Section~\ref{s.LN1}), nontrivial minimal sets in $\CP^n$ do not exist for $n>2$, and so Theorem~\ref{t.S} is (hypothetically) 
nontrivial only for $n=2$. 

\section{Nonexistence of Levi-flats in $\CP^n$, $n\ge 3$.}

\begin{theorem}\label{t.1}
There does not exist a closed nonsingular real analytic Levi-flat hypersurface in $\mathbb{CP}^n$ for $n\ge 3$.
\end{theorem}

There are several published papers with the proofs for $n=2$, but all of them are believed to contain serious gaps.
In this section we outline three different proofs for $n\ge 3$. Two of them are due to Lins Neto~\cite{LN}. The third proof presented here is due to Siu~\cite{Siu}. Siu's result applies to smooth Levi-flat hypersurfaces, which is a much more difficult problem, but in the real analytic category Siu's proof can be simplified.

\subsection{Lins Neto 1.}\label{s.LN1}
Arguing by contradiction, suppose that $M \subset \CP^n$, $n>2$, is a closed real analytic Levi-flat hypersurface. By Prop.~\ref{p.folext} the Levi foliation $\mathcal L$ on $M$ extends to a (nonsingular) holomorphic foliation $\mathcal F$ in an
open neighbourhood $U$ of $M$ in $\CP^n$. By Prop.~\ref{p.st}, $M$ divides $\CP^n$ into two Stein manifolds, call them $X_1$ and $X_2$. The principal step of the proof is to show that $\mathcal F$ extends as a singular holomorphic foliation on each $X_j$. This effectively gives an extension of $\mathcal F$ to $\CP^n$ as a singular holomorphic foliation $\mathcal{\tilde F}$. By Prop.~\ref{p.cod2}, $\mathcal{\tilde F}$ necessarily contains a component $Y\subset\text{Sing}\, \mathcal{\tilde F}$ of codimension 2. Since a Stein manifold cannot contain a compact complex analytic variety of positive dimension, for $n>2$ we conclude that $Y$ cannot be contained in $X_1 \cup X_2$. On the other hand, near $M$ the foliation $\mathcal{\tilde F}$ agrees with $\mathcal F$ and is nonsingular. This is a contradiction that shows that $M$ does not exist. Note that if $n=2$, then 
the singularities of $\mathcal{\tilde F}$ are isolated points, and there is no contradiction. 

It remains to show that $\mathcal F$ extends to $X_j$, $j=1,2$. Note that $X_j \setminus U$ is a compact subset of $X_j$, and  
the extension result for foliations can be considered as a version of the Hartogs theorem (Kugelsatz) on removability of compact singularities for holomorphic functions.  Lins Neto offers the following extension argument. 

Let $X$ be one of the connected components of $\CP^n \setminus M$, and $f: X \to \R$ be a strictly plurisubharmonic exhaustion function. 
After a small $C^2$-perturbation we may further assume that $f$ is also a Morse function. 
Let $X_t = \{ z \in X : f (z) \le t\}$, these are strictly pseudoconvex bounded subdomains of $X$ for almost all $t\in\R$. Then the foliation $\mathcal F$ is defined on $X \setminus X_{t_0}$ for some $t_0>>1$. The idea is to show that if $\mathcal F$ is a holomorphic foliation on 
$X \setminus X_{t_1}$ for some $t_1 \in \R$, then there exists $t_2\in \R$ such that $t_2 < t_1$ and $\mathcal F$ extends as a holomorphic foliation to $X \setminus X_{t_2}$. This will prove the assertion.

Since the level sets $\{f = t\}$ are compact for all $t\in \R$, it suffices to show that for any point $p \in \{f=t_1\}$ there exists $\e >0$ such that the foliation extends to $B(p, \e)$, the ball centred at $p$ of radius $\e$. This can be proved using Hartogs pseudoconvexity, see, e.g.,~\cite{Ran} for relevant results.  Recall that a Hartogs figure in $\C^n$ is defined as 
$$
H_{st} = \left\{z=(z',z_n) \in \C^{n-1}\times \C : |z'|<r', \ |z_n|<1\} \cup \{|z'|<1, \ r_n <|z_n| < 1 \right\},
$$
where $z=(z',z_n)$, and $|z'| < r'$ means $|z_j|< r_j$, $j=1,\dots, n-1$, $r'=(r_1, \dots, r_{n-1})$; $0<r_j<1$, $j=1,\dots,n$.
By the Hartogs Extension theorem, any holomorphic function on $H_{st}$ extends holomorphically to 
the whole polydisc $\Delta=\{|z_j|<1, \ j=1,\dots,n\}$, which is the envelope of holomorphy of $H_{st}$. 
A stronger result, known as the Levi Extension theorem, is that any meromorphic function on $H_{st}$ extends to $\Delta$, 
see, e.g., Siu~\cite{S74}.
This is a special case of a general result stating that the envelope of meromorphy of a domain on a Stein manifold agrees 
with the envelope of holomorphy.
Now, a domain $H \subset X$ is called a Hartogs domain if it is a biholomorphic image of $H_{st}$. 
If $H=\phi(H_{st})$ for some biholomorphic map $\phi :\overline{H_{st}} \to \overline H$, then $\phi$ 
extends holomorphically to $\Delta$, 
and the set $\widehat H = \phi(\Delta)$ is the envelope of holomorphy of $H$, i.e., every function holomorphic 
(resp. meromorphic) on $H$ extends holomorphically (resp. meromorphically) to $\widehat H$. 

If $p$ is a strictly pseudoconvex point of $\{f =t\}$, there exists a local biholomorphic change of coordinates such that 
near $p$ the set $X_t$ is strictly geometrically convex. Then one can directly construct a Hartogs domain $H$ contained 
in $\{f > t\}$ with the property that $p\in \widehat H$. This can be used for the extension of $\mathcal F$ as follows. 
In the local coordinates near $p$, suppose that $\{U_j\}$ is an open cover of $\overline H$ and on each $U_j$ the foliation 
$\mathcal F$ is defined by a holomorphic 1-form
\begin{equation}\label{e.omj}
\omega_j = \sum_{k=1}^n g^j_k dz_k, \ \ g^j_k \in \mathcal O(U_j).
\end{equation}
Let $h_{jk} \in\mathcal O^*(U_j \cap U_k)$ be the compatibility functions so that we have
\begin{equation}\label{e.st}
g^j_k = h_{jl} \,g^l_k,\ \ k = 1, \dots, n .
\end{equation}
Without loss of generality suppose that $g^j_n \not\equiv 0$, so that $f_{jk} :=g^j_k/g^j_n$ is meromorphic in $U_j$. Then~\eqref{e.st} implies that if $U_{j,l} \ne \varnothing$, then $f_{jk} = f_{lk}$ on $U_{j,l}$, so for each $k=1,\dots,n-1,$ 
there exists a meromorphic function $F_k$ on $H$ such that $F_k|_{U_j} = f_{jk}$. By the Levi Extension theorem, 
$F_k$ can be extended to a meromorphic function on $\widehat H$, which contains the point $p$. 
Consider now a meromorphic 1-form on $\widehat H$ given by 
$$
\eta = dz_n + \sum_{k=1}^{n-1} F_k dz_k .
$$
There exists $h \in \mathcal O (\widehat H)$ and a holomorphic 1-form $\omega$ on $\widehat H$ such that $\eta = \frac{1}{h} \omega$. It follows that for all~$j$ we have $\omega|_{U_j} = g_j \omega_j$ for some $g_j \in \mathcal O^*(U_j)$. The form $\omega$ remains integrable and defines a foliation on $\widehat H$. This gives the required extension of $\mathcal F$ to a neighbourhood of $p$. 

The above argument does not discuss the possibility of a multiple-valued extension in case the extension procedure gives overlapping domains. This nontrivial technical issue can be resolved, see Canales Gonz\'alez~\cite{Can} for a detailed step-by-step argument for $n=2$, and Merker-Porten~\cite{MP} for a further in-depth  discussion.

Alternatively, one can argue as follows. A codimension one holomorphic foliation on a complex manifold $X$ of dimension 
$n$ is uniquely determined by an involutive distribution of complex hyper-planes in $TX$. For each $p\in X$, a choice of a hyperplane $H_p \subset T_pX$ defines a point in the Grassmannian $\text{Gr}_{\mathbb C}(n-1,n) \cong \CP^{n-1}$, which can be identified with the projectivization of $T_pX$. Therefore, a holomorphic foliation gives rise to a holomorphic section $s$ 
of the projectivized tangent bundle of $X$. When $X$ is Stein, then through a proper holomorphic embedding of 
$X\hookrightarrow\C^N$ we can represent the tangent bundle as a subbundle of the trivial bundle over $\C^N$, and
so we can view the section $s$ as a holomorphic map into $\CP^{N-1}$ with the values in the projectivized tangent bundle.
In local coordinates the map $s$ is given by the components $g^j_k$ in~\eqref{e.omj}. This is well-defined as seen 
from~\eqref{e.st}.

Therefore, the extension problem of the foliation $\mathcal F$ from $X_j\cap U$ to all of $X_j$ can be considered as a 
problem of extending the holomorphic map $s$ from $X_j \cap U$ into $\CP^{N-1}$ meromorphically to all of~$X_j$.
A general result of this type is proved by Ivashkovish~\cite{Iv}:  a meromorphic map from a domain in a Stein manifold into a compact K\"ahler manifold extends meromorphically to the envelope of holomorphy of the domain. It is based on the previous work of Shiffman~\cite{Sh}, who, in particular, gave a sheaf-theoretic proof of nonbranching of the extended map. The envelope of holomorphy 
of $X_j\cap U$ is all of~$X_j$ and the result follows. It should also be noted, that the meromorphic extension results are considerably simpler when the target is a projective space.

\bigskip

The same argument as above shows, in fact, that no holomorphic foliation $\mathcal F$ on $\CP^n$, $n>2$, admits a nontrivial 
minimal set. Indeed, if $K$ is such a set, then since $K$ is invariant, for any point $p \in K$ there exists a germ of complex hypersurface passing through $p$ and contained in $K$. This, and the result of 
Fujita~\cite{Fu} and  Takeuchi~\cite{Ta}  
implies that $\CP^n \setminus K$ is the union of Stein manifolds. As above, we obtain a contradiction by showing that the restriction of $\mathcal F$ to a connected component of $\CP^n \setminus K$ contains a singularity of dimension at least 1.

\subsection{Lins Neto 2.}
The second proof of nonexistence of Levi-flats in $\CP^n$, $n>2$, given by Lins Neto~\cite{LN}, has two steps: first show 
that the Levi-flat $M\subset \CP^n$ is simply-connected and then apply Haefliger's theorem. The application of Haefliger's 
theorem is straightforward, so the main technical step is to prove that a closed (smooth) Levi-flat hypersurface
$M \subset \CP^n$ is simply-connected when $n>2$. The argument for that, which originates in Bott's~\cite{Bo} proof of 
the Lefschetz Hyperplane theorem, is as follows.

Denote by $X_1$ and $X_2$ the two connected components of $\CP^n \setminus M$, these are Stein manifolds. As such  they 
admit strictly plurisubharmonic exhaustion Morse functions, say, $\psi_1$ and $\psi_2$. It follows from strict plurisubharmonicity that any critical point of $\psi_j$ has index at most $n$, see~\cite[Lem. 3.10.1]{For}. Consider the flow 
$\phi_j: \R \times X_j  \to X_j$, given by the gradient flow of $\psi_j$,  and let $V_j$ be the corresponding vector field on $X_j$, $j=1,2$. Then the singularities of $V_j$ are  points where $d\psi_j =0$, i.e., the Morse singularities of $\psi_j$, and $\psi_j$ is strictly increasing along the nonsingular orbits of the flow. Each singular point $p$ of $V_j$ is hyperbolic, and the stable manifold of $p$, defined by
$$
W^s(p) = \left\{q \in X_j: \lim_{t\to+\infty} \phi_j(t,q)=p \right\} ,
$$
has dimension that is equal to the Morse index of $p$, in particular, does not exceed $n$. Note that since $X_j$ have smooth boundary $M$, there are only finitely many singular points of $V_j$. Let $U$ be a tubular neighbourhood of $M$ that does not contain any singularities of $V_j$. Then for a singular point $p\in X_j$ we may assume that $W^s (p) \subset X_j \setminus \overline U$. After a cutoff we may assume that the flows $\phi_j$ are stationary in $\tilde U \cap X_j$, where 
$\tilde U \subset U$ is a smaller tubular neighbourhood of $M$. This gives a flow $\phi$ on $\CP^n$. Let $W$ be the union 
of all stable manifolds of the flow $\phi$, this is a finite union of manifolds of dimension at most $n$. 

Consider now a loop $\gamma: S^1 \to M$, which we may assume, without loss of generality, to be smooth. 
Since $\tilde U$ is a deformation retract of $M$, for the proof of simple connectivity of $M$ it suffices to show 
that $\gamma$ is contractible in $\tilde U$. Since $\CP^n$ is simply-connected, $\gamma$ is contractible in 
$\CP^n$, and therefore, there exists a smooth disk $D\subset \CP^n$ with boundary $\gamma$. When $n>2$, 
we have 
$$
\dim D + \dim W \le 2 + n < 2n = \dim_{\R} \CP^n .
$$
Therefore, by the Transversality theorem (see, e.g.,~\cite[Ch 3]{Hi}), we may assume that $D$ is chosen in such a way that 
$D \cap W = \varnothing$. We now deform the disc $D$ along the flow $\phi$ constructed above. Since $D$ avoids all stable manifolds, and the exhaustion functions $\phi_j$ are increasing along $\phi$, a deformation of $D$ will be contained in $\tilde U$, and this shows that $M$ is simply-connected. Note that the proof fails when $n=2$ because the transversality does not imply that $D$ can be chosen to be disjoint from $W$. 

To complete the proof of nonexistence, one can now apply Haefliger's theorem~\cite{Ha}. It states that a codimension one smooth real foliation with a null-homotopic closed transversal necessarily has a leaf with one-sided holonomy, in other words, the holonomy group contains a map which is the identity on one side of the transversal, but not on the other. This cannot happen, in particular, if the  map is real-analytic, which is the case if the foliation is real analytic. Now suppose that $M$ is a closed real analytic Levi flat hypersurface in $\CP^n$, $n>2$. Then $M$ is simply connected, and the Levi foliation on $M$ by complex hypersurfaces is real analytic.
Any codimension one foliation on a compact manifold has a closed transversal (see, e.g.,~\cite{CS}), which is null-homotopic 
by simple connectivity of~$M$.  But this contradicts Haefliger's theorem, and so such $M$ does not exist. 

\subsection{Siu} The proof of Siu~\cite{Siu} is technically more involved and applies to the case of $M \subset \CP^n$, $n>2$, with finite smoothness. However, when $M$ is real analytic, certain steps of Siu's argument automatically hold. The proof below is a simplified version of the argument given by Brunella~\cite{Br8}.

Let $M\subset \CP^n$ be a closed real analytic Levi-flat hypersurface, $n>2$. By Prop.~\ref{p.folext}, the Levi foliation 
on $M$ extends to a nonsingular holomorphic foliation $\mathcal F$ in a neighbourhood of $M$. 
Let $N_{\mathcal F}$ be the normal bundle to $\mathcal F$. 
Then $N_{\mathcal F}|_{M}$ is topologically trivial because it has a nonvanishing section---the unit normal to $M$ 
(recall that $M$ is orientable by Prop.~\ref{p.or}). It is well-known that $T^{1,0}\,\CP^n$, equipped with the Fubini-Study metric, is 
Griffiths positive, therefore, $N_{\mathcal F}$, as a quotient bundle of $T^{1,0}\,\CP^n$,  inherits a hermitian metric 
with positive curvature (see Demailly~\cite[Ch.VII,  $\S$6]{D}). Let $\omega$ be the real closed $(1,1)$-form corresponding 
to the curvature, it is exact on $M$, and therefore on its small neighbourhood $U$. Then we can write
$$
\omega|_U = d \beta = (\partial + \overline \partial) (\beta^{0,1} +  \beta^{1,0})=\partial \beta^{0,1}+ 
\overline \partial \beta^{1,0}.
$$
Since $\omega$ is real (i.e., $\omega = \overline{\omega}$), we may assume that 
$\beta^{1,0} = \overline{\beta^{0,1}}$, so that
\begin{equation}\label{e.4}
\omega = \partial \beta^{0,1} + \overline{\partial \beta^{0,1}}.
\end{equation}
Further, $d\beta = (d\beta)^{1,1}$ implies
$$
\overline\partial \beta^{0,1}=0 .
$$
The crucial observation is that on a Stein manifold $X$ of dimension $n>2$ the second cohomology group with coefficients
in $\mathcal O$ and with compact support vanishes, i.e., 
\begin{equation}\label{e.gr}
H^2_{\text{cpt}}(X, \mathcal O) =0 .
\end{equation}
Indeed, by Serre's duality (e.g.,~\cite[Ch.VI, \S7]{D}), this group is isomorphic to $H^{n-2}(X, K_X)$, which is trivial by 
Cartan's theorem B (e.g.,~\cite[Ch.VI, Sec. 6.7]{Ran}) when $X$ is Stein and $n>2$. Here $K_X$ is the canonical bundle.
Now, given a $\overline\partial$-closed form $\beta^{0,1}$ on $U$ we can extend it smoothly to each connected component $X_j$, $j=1,2$, of $\CP^n \setminus M$ to obtain a $(0,1)$-form
on $\CP^n$, not necessarily $\overline\partial$-closed. Denote by $\beta_j$ the restriction of this extension to each $X_j$. The form
$\overline\partial \beta_j$ has compact support in $X_j$, and so by~\eqref{e.gr} there exists compactly supported $h_j$ satisfying 
$\overline\partial h_j = \overline\partial\beta_j$. Then $\beta_j  - h_j$ is a $\overline\partial$-closed $(0,1)$-form in $X_j$ that agrees with $\beta^{0,1}$ near $M$.  This gives us a 
$\overline\partial$-closed extension of $\beta^{0,1}$ to all of $\CP^n$, call it~$\tilde \beta$. 
Since $H^{0,1}_{\overline\partial} (\CP^n) = H^0(\CP^n, \Omega^1) = 0$ (see, e.g.,~\cite[Sec$\,$0.3, p.$\,$48-49]{GH}),
there exists a smooth function $\phi$ on $\CP^n$ such that $\overline\partial \phi = \tilde \beta$. Then, setting
$\psi = i(\overline\phi - \phi)$ and using~\eqref{e.4}, we have
$$
\omega|_U = \partial\overline\partial\,\phi + \overline{\partial\overline\partial\,\phi} =
\partial\overline\partial (\phi - \overline{\phi}) = 
i \,\partial\overline\partial\,\psi.
$$
By the positivity of $\omega$ on $M$, it follows that $\psi $ is a strictly plurisubharmonic function on $M$. Since $M$ is compact, 
$\psi $ attains a maximum at some point $p \in M$. Considering the restriction of $\psi$ to the leaf of the Levi foliation passing 
through $p$, we obtain a contradiction with the Maximum principle for plurisubharmonic functions. This proves that $M$ does not 
exist. 

We see again, that for $n=2$ this proof fails because~\eqref{e.gr} does not hold for a general Stein manifold~$X$.

\section{Some generalizations}

There is vast literature with various generalizations of the nonexistence result by considering a 
wider class of complex manifolds or by lowering the regularity 
of $M$. In this section we discuss only some of these results, primarily those that are directly related to the nonexistence 
problem discussed in this paper. We begin with the work of Ni-Wolfson~\cite{NW}. 

\begin{theorem}[\cite{NW}]
Let $V$ be an irreducible compact K\"ahler manifold of complex
dimension $n$ with nonnegative holomorphic bisectional curvature of complex
positivity $\ell$. Then there are no real analytic Levi flat submanifolds of dimension
$m$ in $V$ when $m \ge 2 (n+1) - \ell$.
\end{theorem}

The complex positivity $\ell$ is defined as follows. The symmetric bilinear form 
$H_Y(W,Z)= \langle R(Y, JY)W, JZ\rangle, $ where $Y,W,Z \in TV$, is positive-semidefinite for $Y\ne 0$. 
Let $\mathcal N_Y$ be its null-space and $n(Y)$ be its dimension. Then the complex positivity of $V$ is
$\ell := \inf_{Y\ne 0} (n - n(Y))$. In particular, for $V=\CP^n$, $\ell = n$, and we obtain $m \ge n+2$, 
which for $n>2$ covers the case of real hypersurfaces, but for $n=2$ gives no information. The idea of the 
proof is similar to that of~\cite{LN}: show that $M$ is simply-connected and then use Haefliger's theorem. 

In~\cite{Siu} Siu gives the proof of nonexistence of smooth 
Levi-flat hypersurfaces in~$\CP^n$, $n>2$. Siu's original proof works for $C^{12}$-smooth $M$, but further improvements 
of the regularity of $M$ are possible, for example, Cao-Shaw~\cite{CSh} proved that there does not exist Lipschitz 
Levi-flat hypersurfaces in $\CP^n$, $n>2$, and Iordan-Matthey~\cite{IM} proved nonexistence for $M$ of class 
Sobolev $W^s$, for $s>9/2$.

To generalize the class of ambient manifolds, Ohsawa~\cite{Oh7} considered the case of compact K\"ahler 
manifolds. Using $\overline\partial$-methods and Hodge theory, he proved the following.

\begin{theorem}
Let $X$ be a compact K\"ahler manifold of dimension $n \ge 3$ and let $M$ 
be a real analytic Levi-flat in $X$ . Then, $ X \setminus M$ does not admit
any $C^2$ plurisubharmonic exhaustion function whose Levi form has at least
3 positive eigenvalues outside a compact subset of $X \setminus M$. In particular,
$X \setminus M$ is not a Stein manifold.
\end{theorem}
This, in particular, implies that the examples of Ohsawa~\cite{Oh82} and Nemirovski~\cite{Ne} of Levi-flats with Stein 
complement cannot exist on compact K\"ahler manifolds of dimension $\ge 3$. 

Another generalization was obtained by Brinkschulte~\cite{Bri}.

\begin{theorem}[\cite{Bri}]\label{t.bri}
Let $X$ be a complex manifold of dimension $n\ge 3$. Then there does
not exist a smooth compact Levi-flat real hypersurface $M$ in $X$ such that the normal
bundle to the Levi foliation admits a Hermitian metric with positive curvature along
the leaves.
\end{theorem}

This result is based on the earlier work of Brunella~\cite{Br8}, who was the first to emphasize the crucial role 
of the positivity of the normal bundle of the foliation, and Ohsawa~\cite{Oh}, who conjectured this result 
in~\cite[Sec.5.1.1]{Oh18}. 
If in Theorem~\ref{t.bri} we take 
$X=\CP^n$ and $M$ a Levi-flat with the Levi foliation $\mathcal L$, then since $\CP^n$ admits a hermitian metric of positive curvature, we obtain a positive 
metric on $\mathcal N\mathcal L$. This contradicts the theorem, and therefore such $M$ does not exist. 
Thus, this result generalizes the nonexistence of Levi-flats in $\CP^n$, $n>2$, in the smooth category.
It should be also noted
that Brunella~\cite[Ex. 4.2]{Br8}, see also Ohsawa~\cite{Oh18}, gave an example showing that Theorem~\ref{t.bri} 
cannot hold for $n=2$ even when $X$ is compact K\"ahler. It was proved in ~\cite{I23} (the appendix is written jointly 
by Iordan and Lempert) that the Levi-flats arising in this example cannot be embedded into $\CP^n$.

In a different direction, an interesting generalization of nonexistence result in $\CP^n$ was obtained by Sargsyan~\cite{Sar}.

\begin{theorem}[\cite{Sar}]
There exists no smooth real hypersurface in $\CP^n$, $n \ge 3$, whose
Levi form has constant signature and satisfies one of the following conditions:

(i) the Levi form has at least two zero eigenvalues;

(ii) the Levi form has at least one zero eigenvalue and two eigenvalues of opposite signs.
\end{theorem}

Concerning minimal sets, Adachi and Brinkschulte~\cite{AB3} proved the following result, which confirms a conjecture by 
Brunella~\cite{Br8}. This result, in particular, implies that the foliation $\mathcal F$ does not admit any nontrivial minimal sets.

\begin{theorem}[\cite{AB3}]
Let $X$ be a compact complex manifold of dimension $\ge 3$. Let $\mathcal F$
be a codimension one holomorphic foliation on $X$ with ample normal bundle $\mathcal{NF}$ .
Then every leaf of $\mathcal F$ accumulates to $\text{Sing}\,\mathcal F$.
\end{theorem}

Note that the theorem above is a generalization of Lins Neto’s  result on nonexistence of minimal sets in $\CP^n$, since
there the normal bundle of a holomorphic foliation is automatically positive. The approach in \cite{AB3}
and the follow-up paper by Adachi, Biard and Brinkschulte~\cite{ABB} combines Ohsawa’s and Siu’s ideas. The
first step is to produce a holomorphic connection on the normal bundle (this was implicit in \cite{Oh7}), and then
to obtain a contradiction similar to Siu by solving the $\partial\overline\partial$-problem for the difference of 
the curvatures
of the Chern connection and the holomorphic connection.

While the problem of nonexistence of Levi-flats $M$ in $\CP^2$ remains open, there are some partial answers to 
the question. For example,
Adachi-Brinkschulte~\cite{AB} gave restriction on the totally real Ricci
curvature of $M$, Mishchenko~\cite{Mi2} proved that a hypothetical Levi-flat cannot have a commutative first fundamental group $\pi_1(M)$, and Inaba and Mishchenko~\cite{IMi} showed that $\pi_1(M)$ cannot have nonexponential growth. 
Finally, Deroin and Dupont~\cite{DD} investigate Levi-flat hypersurfaces in complex surfaces. 

\bigskip

{\bf Acknowledgment.} The author would like to thank C. Denson Hill, L\'aszl\'o Lempert, Stefan Nemirovski, and Dror Varolin for valuable discussions concerning the content of this paper. The author is partially supported by the Natural Sciences and Engineering Research
Council of Canada.


\end{document}